\documentclass[11pt]{article}

\usepackage{url,amsmath,amsthm,amscd,amsfonts,graphicx,psfrag,marginnote,hyperref}
\usepackage[all]{xy}
\usepackage{color}
\usepackage{fancyhdr}
\usepackage{wrapfig}

\newtheoremstyle{stuff}{\topsep}{\topsep}%
     {\itshape}%         Body font
     {}%         Indent amount (empty = no indent, \parindent = para indent)
     {\bfseries}% Thm head font
     {}%        Punctuation after thm head
     {.5em}%     Space after thm head (\newline = linebreak)
     {\thmnote{#3}}%         Thm head spec

\theoremstyle{plain}
\newtheorem{thm}{Theorem}

\newtheorem{lem}[thm]{Lemma}

\theoremstyle{remark}
\newtheorem*{rem}{Remark}

\theoremstyle{definition}
\newtheorem{defn}{Definition}

\theoremstyle{stuff}

\numberwithin{equation}{section}
 
\usepackage[scale=0.72]{geometry}

\usepackage{sectsty}
\sectionfont{\center \sc \large}
\subsectionfont{\bf \large}
\subsubsectionfont{\it \large}

\usepackage{setspace}

\newcommand{\bea}{\begin{eqnarray}}
\newcommand{\eea}{\end{eqnarray}}
\newcommand{\be}{\begin{equation}}
\newcommand{\ee}{\end{equation}}
\begin{document}

%titlepage
\pagestyle{empty}

\begin{flushright}
\begin{tabular}{l}
 \\
\end{tabular}
\end{flushright}

{\noindent \Large \bf\fontfamily{pag}\selectfont  The remodeling conjecture and the Faber-Pandharipande\\[0.6em] formula}\\

\vspace*{0.2cm}
\noindent \rule{\linewidth}{0.5mm}

\vspace*{0.8cm}

{\noindent \fontfamily{pag}\selectfont  Vincent Bouchard$^\dagger$, Andrei Catuneanu$^\dagger$, Olivier Marchal$^\dagger$ and Piotr Su\l kowski$^\ddagger$}\\[2em]
{\small {\it \indent $^\dagger$ Department of Mathematical and Statistical Sciences\\
\indent University of Alberta\\
\indent 632 CAB, Edmonton, Alberta T6G 2G1\\
\indent Canada}\\
\indent \url{vincent@math.ualberta.ca}\\
\indent \url{catunean@ualberta.ca}\\
\indent \url{olivier.marchal@polytechnique.org}}\\[0.3em]

{{\small {\it \indent $^\ddagger$ California Institute of Technology\\
\indent Pasadena, CA 91125\\
\indent USA\\[0.2em]
\indent Faculty of Physics, University of Warsaw\\
\indent ul. Ho{\.z}a 69, 00-681 Warsaw\\
\indent Poland}\\
\indent \url{psulkows@theory.caltech.edu}}\\[0.3em]

\vspace*{0.8cm}

\hspace*{1cm}
\parbox{11.5cm}{{\sc Abstract:} In this note, we prove that the free energies $F_g$ constructed from the Eynard-Orantin topological recursion applied to the curve mirror to $\mathbb{C}^3$ reproduce the Faber-Pandharipande formula for genus $g$ Gromov-Witten invariants of $\mathbb{C}^3$. This completes the proof of the remodeling conjecture for $\mathbb{C}^3$. }
\pagebreak

\pagestyle{fancy}
% table of contents
%\pagenumbering{roman}
\pagenumbering{arabic}

\tableofcontents 

% beginning of notes

%\pagebreak
%

\section{Introduction}

According to the ``remodeling conjecture'' \cite{Bouchard:2009, Marino:2008}, the generating functions for Gromov-Witten invariants of a toric Calabi-Yau threefold $X$ can be computed by applying the Eynard-Orantin  topological recursion \cite{Eynard:2007,Eynard:2008} to the family of complex curves mirror to $X$. More precisely, the Eynard-Orantin topological recursion produces an infinite tower of meromorphic differentials $W^g_n$, which are mapped by the open/closed mirror map to generating functions of open Gromov-Witten invariants of $(X,L)$, where $L$ is an appropriate Lagrangian submanifold of $X$. The recursion also produces an infinite tower of free energies $F_g$ that are mapped by the closed mirror map to generating functions of closed Gromov-Witten invariants of $X$.

The simplest case to consider is when $X = \mathbb{C}^3$. In this case, open Gromov-Witten invariants can be computed using the topological vertex formalism \cite{Aganagic:2005,Li:2009}. The generating functions $W^g_n$ for open Gromov-Witten invariants can be written in terms of Hodge integrals. For this particular geometry, it was proved independently by Chen and Zhou that the ``open part'' of the remodeling conjecture is true \cite{Chen:2009, Zhou:2009}, namely, that the meromorphic differentials $W^g_n$ constructed from the Eynard-Orantin topological recursion applied to the curve mirror the $\mathbb{C}^3$ indeed reproduce the open Gromov-Witten generating functions (see also \cite{Zhu:2010}).

To complete the proof of the remodeling conjecture for $\mathbb{C}^3$, it remains to be proved that the free energies $F_g$ reproduce the closed Gromov-Witten invariants of $\mathbb{C}^3$. The only non-zero Gromov-Witten invariants of $\mathbb{C}^3$ correspond to constant maps, and have been computed many years ago by Faber and Pandharipande \cite{Faber:2000}, giving the well known result (for $g \geq 2$):
\begin{equation} \label{Final}
F_g = (-1)^g \frac{|B_{2g}| |B_{2g-2}|}{2(2g) (2g-2) (2g-2)!},
\end{equation}
where $B_{n}$ is the $n$'th Bernoulli number.\footnote{We define the Bernoulli numbers through the generating function:
$$ %\begin{equation}
\frac{t}{\mathrm{e}^t-1} = \sum_{m=0}^\infty B_m \frac{t^m}{m!}.
$$ %\end{equation}
}
$F_g$ is the Gromov-Witten invariant for constant maps from genus $g$ Riemann surfaces to $\mathbb{C}^3$. We also recall that the above free energies arise in the $\lambda$ expansion of $M(q)^{1/2} = \exp \big( \sum_{g=0}^{\infty} \lambda^{2g-2} F_g \big)$, where $M(q)$ is the MacMahon function
\begin{equation}\label{eq:macmah}
M(q) = \prod_{k=1}^\infty \left(1-q^k \right)^{-k},
\end{equation}
and $q = \mathrm{e}^{\mathrm{i} \lambda}$. 

As part of a broader study of Gromov-Witten invariants for constant maps from the point of view of the topological recursion, it was conjectured in \cite{BS:2011} that the free energies $F_g$ computed from the Eynard-Orantin topological recursion applied to the curve mirror to $\mathbb{C}^3$ indeed reproduce the Faber-Pandharipande formula. The conjecture was checked computationally up to genus $7$.  Our main theorem in this note is a proof of this conjecture, thus completing the proof of the remodeling conjecture for $\mathbb{C}^3$. Our proof relies on the previous work of Chen and Zhou \cite{Chen:2009,Zhou:2009} where the open part of the remodeling conjecture for $\mathbb{C}^3$ is proved. The starting point can also be seen as a particular case of the recent work of Eynard in \cite{Eynard:2011}, as explained in Appendix \ref{appA}.

\begin{rem}
Given a matrix model, the topological recursion applied to its spectral curve should reproduce the free energies of the matrix model. So one might be tempted to use this approach to prove the theorem in this paper, by constructing a matrix model for $M(q)^{1/2}$ and computing its spectral curve. However, as discussed in \cite{BS:2011}, the constant contributions to the free energies are subtle, and the recursion may produce results that differ from the matrix model; hence matrix models cannot really be used to prove our main theorem. We illustrate this issue in Appendix \ref{appB} for the case of $\mathbb{C}^3$ studied in this paper.
\end{rem}

\begin{rem}
After completion of this work, we were informed of an independent proof of our main theorem by Shengmao Zhu, using very similar ideas \cite{Zhu:2011}. 
\end{rem}

\subsection*{Outline}

We review the fundamentals of the Eynard-Orantin topological recursion in Subsection \ref{s:EO} and the remodeling conjecture in Subsection \ref{s:remodeling}. We then specialize to the $\mathbb{C}^3$ geometry in Subsection \ref{s:C3}, describing the mirror geometry and the statement of Chen and Zhou's theorem in terms of Hodge integrals. Section \ref{s:proof} is devoted to the proof of our main theorem. In Section \ref{s:conclusion} we conclude with a few comments. We discuss in Appendix \ref{appA} the relation between Chen and Zhou's theorem and the recent work of Eynard \cite{Eynard:2011}. In Appendix \ref{appB} we discuss the relation with matrix models alluded to above.

\subsection*{Acknowledgments}
We would like to thank Renzo Cavalieri, Bertrand Eynard, Melissa Liu, Nicolas Orantin and Jian Zhou for enjoyable discussions. 
The research of V.B. and O.M. are supported by a University of Alberta startup grant and an NSERC Discovery grant. The research of A.C. is supported by an NSERC Undergraduate Student Research Award.
The research of P.S. is supported by the DOE grant DE-FG03-92ER40701FG-02 and the European Commission under the Marie-Curie International Outgoing Fellowship Programme.

\section{Background}

\label{s:background}

\subsection{Eynard-Orantin topological recursion}

\label{s:EO}

In this paper we prove that the remodeling conjecture is true for the free energies $F_g$ constructed from the mirror curve to $\mathbb{C}^3$. The remodeling conjecture is based on the Eynard-Orantin topological recursion \cite{Eynard:2007, Eynard:2008}. In this section we define the Eynard-Orantin topological recursion.

\subsubsection{Ingredients}

We start with a smooth complex curve
\begin{equation}
C = \{ H(x,y) = 0 \}
\end{equation}
in $\mathbb{C}^2$ or $(\mathbb{C}^*)^2$ often called ``spectral curve".
It defines a non-compact Riemann surface, which we also denote by $C$. $x,y: C \to \mathbb{C}$ are holomorphic functions on $C$. 

We assume that the map $x: C \to \mathbb{C}$ has only simple ramification points. In this paper we focus on the case with a single ramification point. Let $a \in C$ be the ramification point of $x$. Locally near $a$ the map is a double-sheeted covering, hence we have a deck transformation map
\begin{equation}
s: U \to U
\end{equation}
which is defined locally in a neighborhood $U$ of $a$. The deck transformation map means that
\begin{equation}
x(t) = x(s(t))
\end{equation}
for some local coordinate $t$ near $a$.

The type of objects that we will be interested in are meromorphic symmetric differentials on $C^n$. In local coordinates $z_i := z(p_i)$, $p_i \in C$, $i=1,\ldots,n$ a degree $n$ differential can be written as\footnote{For simplicity we will omit the tensor product symbol $\otimes$ between the differentials.}
\begin{equation}
W_n(p_1, \ldots, p_n) = w_n(z_1, \ldots, z_n) \mathrm{d} z_1 \cdots \mathrm{d} z_n,
\end{equation}
where $w(z_1, \ldots, z_n)$ is meromorphic in each variable.

To initialize the recursion we need to introduce a particular degree $2$ differential.
\begin{defn}\label{d:bergman}
We define $W_2^0(p_1,p_2)$ to be the \emph{fundamental normalized bi-differential}  \cite[p.20]{Fay:1970} which is uniquely defined by the conditions:
\begin{itemize}
\item It is symmetric, $W_2^0(p_1, p_2) = W_2^0(p_2,p_1)$;
\item It has its only pole, which is double, along the diagonal $p_1 = p_2$, with no residue; its expansion in this neighborhood has the form
\begin{equation}
W_2^0(p_1, p_2) = \left(\frac{1}{(z_1-z_2)^2} + \text{regular} \right) \mathrm{d} z_1 \mathrm{d} z_2;
\end{equation}
\item It is normalized by requiring that its periods about a basis of $A$-cycles on $C$ vanish.\footnote{$W_2^0(p_1,p_2)$ has also been called \emph{Bergman kernel} in the literature. It is the second order derivative of the $\log$ of the prime-form on $C$ \cite{Fay:1970}.}
\end{itemize}
\end{defn}

Having now defined the main ingredients, we can introduce the Eynard-Orantin recursion, following \cite{Eynard:2007, Eynard:2008}.

\subsubsection{The Eynard-Orantin topological recursion}

Let $\{ W^g_n \}$ be an infinite sequence of meromorphic differentials $W^g_n(p_1, \ldots, p_n)$ for all integers $g\geq 0$ and $n > 0$ satisfying the condition $2g - 2 + n \geq 0$. We say that the differentials with $2g-2+n > 0$ are \emph{stable}; $W_2^0(p_1, p_2)$ is the only unstable differential.

Let us introduce the shorthand notation $S = \{p_1, \ldots, p_{n} \}$. Then:
\begin{defn}\label{d:recursion}
We say that the meromorphic differentials $W_n^g$ satisfy the \emph{Eynard-Orantin topological recursion} (for $x$ with a single ramification point $a$) if:
\begin{equation}
W^g_{n+1}(p_0, S) =  \underset{q=a}{\text{Res}}\,  K(p_0, q) \Big(W^{g-1}_{n+2} (q, s(q), S) +\sum_{\substack{g_1 + g_2 = g \\ I \sqcup J = S}} W^{g_1}_{|I|+1}(q, I) W^{g_2}_{|J|+1}(s(q), J) \Big),
\end{equation}
where $K(p_0,q)$ is the \emph{Eynard kernel} defined below. The recursion here is on the integer $2g - 2 + n$, which is why it is called a topological recursion. The initial condition of the recursion is given by the unstable $W_2^0(p_1,p_2)$ defined above.
\end{defn}

\begin{defn}\label{d:kernel}
The \emph{Eynard kernel} $K(p_0,q)$ is defined, in local coordinate $q$ near $a$, by
\begin{equation}\label{eq:eynard}
K(p_0,q) = \frac{1}{2} \frac{\int_q^{s(q)} W^0_2(p_0,q')}{\omega(q) - \omega(s(q))},
\end{equation}
where $\omega(q)$ is the meromorphic one-form $\omega(q) = y(q) \mathrm{d} x(q)$ if the curve $C$ is in $\mathbb{C}^2$, and $\omega(q) = \log y(z) \frac{\mathrm{d} x(q)}{x(q)}$ if the curve $C$ is in $(\mathbb{C}^*)^2$. Here, $\frac{1}{\mathrm{d} x(q)}$ is the contraction operator with respect to the vector field $\left( \frac{\mathrm{d} x}{\mathrm{d} q} \right)^{-1} \frac{\partial}{\partial q}$.
\end{defn}

Definitions \ref{d:bergman}, \ref{d:recursion} and \ref{d:kernel} together define the Eynard-Orantin topological recursion for the curve $C$. (We refer the reader to \cite{Eynard:2007, Eynard:2008} for additional details and properties). 

\subsubsection{The $F_g$'s}

We can also extend the construction to $n = 0$ objects, $F_g := W_0^g$, which are just numbers. Those are the objects that we will concentrate on in this paper. To construct the $F_g$, $g \geq 2$ (the stable ones), we need an auxiliary equation. Let us first define
\begin{equation}\label{eq:primitive}
\Phi(q) = \int_0^q \omega(q'),
\end{equation}
which is the primitive of the one-form $\omega(q)$ for an arbitrary base point $0$. We then define:
\begin{defn}\label{d:fg}
The numbers $F_g$, $g \geq 2$, are constructed from the one-forms $W_1^g(p)$ by:
\begin{equation}\label{eq:fg}
F_g = \frac{(-1)^g}{2-2g}  \underset{q=a}{\text{Res}}\, \Phi(q) W_1^g(q).
\end{equation}
\end{defn}

\begin{rem}
Note that as in \cite{BS:2011} we introduce a factor of $(-1)^g$ in the definition of the  $F_g$ which is absent in the original formalism \cite{Eynard:2007}. As explained in \cite{BS:2011} (p.12), this factor is required to make precise comparison with results in Gromov-Witten theory due to different normalizations of the string coupling constant.
\end{rem}

\begin{rem}
As explained by Eynard and Orantin, Definition \ref{d:fg} is the $n=0$ extension of the relation:
\begin{equation}W_{n}^g(p_1,\dots,p_n)=\frac{1}{2-2g-n}\underset{q=a}{\text{Res}}\, \Phi(q)W_{n+1}^g(q,p_1,\dots,p_n).
\end{equation}  
\end{rem}

To summarize, given an affine curve $C$, the Eynard-Orantin topological recursion constructs an infinite tower of meromorphic differentials $W_n^g(p_1, \ldots, p_n)$ (Definition \ref{d:recursion}), and numbers $F_g := W_0^g$ (Definition \ref{d:fg}), for $g \geq 0$, $n >0$, satisfying the stability condition $2g-2+n > 0$. The recursion kernel is the Eynard kernel (Definition \ref{d:kernel}), and the initial condition of the recursion is the fundamental normalized bi-differential on $C$ (Definition \ref{d:bergman}).

\subsection{The remodeling conjecture}

\label{s:remodeling}

The remodeling conjecture  \cite{Bouchard:2009, Marino:2008} is an application of the Eynard-Orantin recursion in the world of Gromov-Witten theory and mirror symmetry. Roughly speaking, the statement of the conjecture is the following. We consider Gromov-Witten theory of a toric Calabi-Yau threefold $X$. The mirror theory lives on a family of complex curves, known as the \emph{mirror curve}, living in $(\mathbb{C}^*)^2$. We can apply the Eynard-Orantin recursion to the mirror curve to compute a tower of meromorphic differentials $W_g^n$ and free energies $F_g$. The statement of the remodeling conjecture is then that the $W_g^n$ are mapped by the open/closed mirror maps to the genus $g$, $n$-hole generating functions of open Gromov-Witten invariants, while the $F_g$ are mapped by the closed mirror map to the genus $g$ generating functions of closed Gromov-Witten invariants. For more details on this conjecture and on the geometry of mirror symmetry, we refer the reader to  \cite{Bouchard:2009, Bouchard:2010, BS:2011,Marino:2008} and subsequent work.

One aspect of the remodeling conjecture was clarified in \cite{BS:2011}: the issue of constant maps. The simplest Gromov-Witten invariants of $X$ are given by constant maps from closed Riemann surfaces. On the mirror side, those should correspond to the constant term in the free energies $F_g$ computed by the recursion.\footnote{Note that in this setup we are considering a family of curves, hence the $F_g$ are functions of the parameters of this family.} In \cite{BS:2011} it was argued that the remodeling conjecture also holds for constant maps. Two conjectures were formulated; the first conjecture stated that the $F_g$ obtained from the mirror curve to the simplest toric Calabi-Yau threefold $X = \mathbb{C}^3$ give the correct Gromov-Witten invariants for constant maps to $\mathbb{C}^3$; the second conjecture stated that the $F_g$ obtained from the mirror curve to a general toric Calabi-Yau threefold $X$ is equal to $\chi(X)$ times the $F_g$ of $\mathbb{C}^3$, as expected from Gromov-Witten theory. In this paper we prove the first conjecture about constant maps to $\mathbb{C}^3$.

\subsection{Gromov-Witten theory of $\mathbb{C}^3$}

\label{s:C3}

We now focus on a particular smooth curve, the mirror curve to Gromov-Witten theory of $\mathbb{C}^3$. We refer the reader to \cite{Bouchard:2009,BS:2011,Hori:2000} for more details on how to construct the mirror curve to a particular toric Calabi-Yau threefold. Note that for this geometry, the mirror curve is really a curve, and not a family of curves; hence the $F_g$ are really just numbers. This is because the only Gromov-Witten invariants of $\mathbb{C}^3$ correspond to constant maps, since there are no compact cycles in $\mathbb{C}^3$.

The (framed) mirror curve to $\mathbb{C}^3$ is given by the smooth complex curve
\begin{equation}\label{eq:mirrorc3}
C = \{ x - y^f + y^{f+1} = 0 \} \subset (\mathbb{C}^*)^2,
\end{equation}
where the framing $f \in \mathbb{Z}$ is taken to be generic (\emph{i.e.} not $0$ or $-1$ --- see \cite{BS:2011} for clarifications on the issue of framing). $C$ is a genus $0$ curve with three punctures.
\begin{rem}
Note that this curve is related to the curve in \cite{BS:2011} by the transformation $(x,y) \mapsto ((-1)^{f+1} x, -y)$, which does not change the $F_g$. Here we use the conventions of \cite{Chen:2009,Zhou:2009} so that we can use their theorem directly for the correlation functions. Henceforth we follow the notation of \cite{Chen:2009}.
\end{rem}

We introduce the parametrization
\begin{equation}\label{Para1}
y(t) = \frac{1}{f+1}\left( \frac{1}{t} + f \right), \qquad x(t) = y(t)^f(1-y(t)).
\end{equation}
The $x$-projection has a single ramification point, which is at $t = \infty$ (\emph{i.e.} $y=\frac{f}{f+1}$). 

Following \cite{Chen:2009,Zhou:2009} (based on the remodeling conjecture in \cite{Bouchard:2009, Bouchard:2009ii}), we introduce the following functions on $C$, for $0 \leq b \in \mathbb{Z}$:
\begin{align} \label{Phib}
\phi_b(t) =& \left ( x(t) \frac{\mathrm{d}}{\mathrm{d} x(t)} \right)^b \left(\frac{t-1}{f+1} \right)\\
=& \left(\frac{t (t-1) (f t+1)}{f+1} \frac{\mathrm{d}}{\mathrm{d}t}\right)^b \left(\frac{t-1}{f+1} \right).
\end{align}
Note that $\phi_b(t)$ is a polynomial in $t$ of degree $2b+1$. We also introduce
\begin{equation}
\phi_{-1}(t) = - \log \left( 1 + \frac{1}{f t} \right),
\end{equation}
so that
\begin{equation}
\phi_0(t) =  \left(\frac{t (t-1) (f t+1)}{f+1} \frac{\mathrm{d}}{\mathrm{d}t}\right) \phi_{-1}(t).
\end{equation}
Remark that
\begin{equation}\label{eq:logy}
\log y(t) = \log \left(\frac{1}{f t} + 1 \right) + \log \frac{f}{f+1}= - \phi_{-1}(t) + \log \frac{f}{f+1}.
\end{equation}
We also introduce the corresponding one-form:
\begin{equation}
\zeta_b(t) = \mathrm{d} \phi_b(t).
\end{equation}

The ``open'' statement of the remodeling conjecture relates the correlation functions $W^g_n$ constructed from the Eynard-Orantin recursion applied to the framed mirror curve \eqref{eq:mirrorc3} to generating functions of open Gromov-Witten invariants to $\mathbb{C}^3$. It is well known that those can be rewritten in terms of Hodge integrals, following the topological vertex formalism \cite{Aganagic:2005, Li:2009}. So to state the open part of the conjecture we need to introduce standard notation for Hodge integrals. 

Let $\overline{M}_{g,h}$ be the Deligne-Mumford compactification of the moduli space of complex algebraic curves of genus $g$ with $h$ marked points. Let $\mathbb{E}$ be the Hodge bundle on $\overline{M}_{g,h}$. We define the $\lambda_i$ classes as the Chern classes of $\mathbb{E}$:
\begin{equation}
\lambda_i = c_i(\mathbb{E}) \in H^{2i}(\overline{M}_{g,h}; \mathbb{Q}).
\end{equation}
As usual, we define the generating series:
\begin{equation}
\Lambda^\vee_g(t) = \sum_{i=0}^g (-1)^i \lambda_i t^{g-i}.
\end{equation}
We also define the $\psi_i$ class as the first Chern class of the cotangent line bundle $\mathbb{L}_i$ at the $i^{\text{th}}$ marked point:
\begin{equation}
\psi_i = c_1(\mathbb{L}_i) \in H^2(\overline{M}_{g,h}; \mathbb{Q}).
\end{equation}
Hodge integrals are intersection numbers of $\lambda_j$ and $\psi_i$ classes:
\begin{equation}
\langle \psi_1^{j_1} \cdots \psi_h^{j_h} \lambda_1^{k_1} \cdots \lambda_g^{k_g} \rangle := \int_{\overline{M}_{g,h}} \psi_1^{j_1} \cdots \psi_h^{j_h} \lambda_1^{k_1} \cdots \lambda_g^{k_g}.
\end{equation}
Of course, since the dimension of $\overline{M}_{g,h}$ is $3g-3+h$, the Hodge integrals are non-vanishing only if
\begin{equation}
j_1 + \ldots + j_h + \sum_{i=1}^g i k_i = 3g - 3 + h.
\end{equation}
In the following we also use Witten's notation for Hodge integrals:
\begin{equation}
\langle \tau_{b_1} \cdots \tau_{b_n} \cdots \rangle = \langle \psi_1^{b_1} \cdots \psi_n^{b_n} \cdots \rangle.
\end{equation}

We are now ready to state the ``open'' part of the remodeling conjecture, rewritten in terms of Hodge integrals. The ``open'' statement was proved in \cite{Chen:2009,Zhou:2009}:
\begin{thm}[Chen \cite{Chen:2009}, Zhou \cite{Zhou:2009}]\label{thm:CZ}
The correlation functions $W^g_n$ produced by the Eynard-Orantin recursion applied to the framed mirror curve to $\mathbb{C}^3$ \eqref{eq:mirrorc3} are given by the meromorphic differentials:\footnote{Note that we have an extra factor of $(-1)^{2g-2+n} = (-1)^n$ with respect to the formula in \cite{Chen:2009,Zhou:2009}; this is due to the fact that our Eynard kernel, Definition \ref{d:kernel}, has  a minus sign difference with the kernel used in \cite{Chen:2009,Zhou:2009}.}
\begin{equation}
W^g_n(t_1, \ldots, t_n) = (-1)^{g} (f(f+1))^{n-1} \sum_{b_1, \ldots, b_n \geq 0} \langle \tau_{b_1} \cdots \tau_{b_n} \Gamma_g(f) \rangle \prod_{i=1}^n \zeta_{b_i} (t_i),
\end{equation}
where we introduced the notation
\begin{equation}
\Gamma_g(f) = \Lambda^\vee_g(1) \Lambda^\vee_g(f) \Lambda^\vee_g (-f-1).
\end{equation}
\end{thm}
This theorem, first conjectured in \cite{Bouchard:2009, Bouchard:2009ii, Marino:2008}, was proved in \cite{Chen:2009, Zhou:2009} by using the symmetrized cut-and-join equation as in the mathematical theory of the topological vertex in Gromov-Witten theory \cite{Li:2009}. The line of reasoning is similar to what was used by Eynard, Mulase and Safnuk \cite{Eynard:2009} to prove the remodeling conjecture for Hurwitz numbers \cite{Bouchard:2009ii}.
Note that Theorem \ref{thm:CZ} is also a consequence of a more general formalism recently developed by Eynard in \cite{Eynard:2011}. We will say more about that in the Appendix \ref{appA}.

In this paper we complete the proof of the remodeling conjecture for $\mathbb{C}^3$ by proving that the free energies $F_g$ also produce the correct Gromov-Witten invariants, namely closed Gromov-Witten invariants for constant maps to $\mathbb{C}^3$.

\section{The free energies for $\mathbb{C}^3$}

\label{s:proof}

The main result of this paper is the following theorem:
\begin{thm}\label{thm:main}
The free energies $F_g$, $g \geq 2$ (defined in \eqref{eq:fg}) produced by the Eynard-Orantin recursion applied to the framed mirror curve to $\mathbb{C}^3$ \eqref{eq:mirrorc3} are given by:
\begin{equation}
F_g = (-1)^g \frac{|B_{2g}| |B_{2g-2}|}{2(2g) (2g-2) (2g-2)!},
\end{equation}
where $B_n$ is the $n^{\text{th}}$ Bernoulli number.
This is the Gromov-Witten invariant for constant maps from genus $g$ Riemann surfaces to $\mathbb{C}^3$, as proved in \cite{Faber:2000}.
\end{thm}
This completes the proof of the remodeling conjecture for $\mathbb{C}^3$. This theorem was conjectured in \cite{BS:2011}, where it was shown to hold computationally up to genus $7$. Here we provide a proof based on the results of Chen and Zhou \cite{Chen:2009, Zhou:2009}.

\begin{rem}
Note that it follows from Theorem \ref{thm:main} that the free energies are ``framing-independent'', \emph{i.e.} do not depend on the framing $f$, while the correlation functions in Theorem \ref{thm:CZ} do depend on $f$. This result confirms ``symplectic invariance'' (as clarified in \cite{BS:2011}) of the free energies (but not of the correlation functions), as expected from the work of Eynard and Orantin \cite{Eynard:2007,Eynard:2008}. 
\end{rem}

Let us start by proving the intermediate lemma:
\begin{lem}\label{lem}
The free energies $F_g$,  $g \geq 2$ (defined in \eqref{eq:fg}) produced by the Eynard-Orantin recursion applied to the framed mirror curve to $\mathbb{C}^3$ \eqref{eq:mirrorc3} are given in terms of Hodge integrals by:
\begin{equation}
F_g =  \frac{1}{(2-2g) f(f+1)} \langle \psi_1 \Gamma_g(f) \rangle.
\end{equation}
\end{lem}

\begin{proof}
What we need to do is evaluate the residue in \eqref{eq:fg} for our particular curve \eqref{eq:mirrorc3}:
\begin{equation}
F_g = \frac{(-1)^g}{2-2g}  \underset{q=a}{\text{Res}}\, \Phi(q) W_1^g(q),
\end{equation}
with
\begin{equation}
\Phi(q) = \int_0^q \log y(q') \frac{\mathrm{d} x(q')}{x(q')} = \int \omega
\end{equation}
an arbitrary primitive of the one-form $\omega = \log y \frac{\mathrm{d} x}{x}$.

First, note that we can integrate by part, and rewrite instead
\begin{align}
F_g =& \frac{(-1)^g}{2-2g}  \underset{q=a}{\text{Res}} \left( \left( \int \omega(q) \right) W_1^g(q) \right)\\
=& -  \frac{(-1)^g}{2-2g}  \underset{q=a}{\text{Res}}\left(\omega(q) \left( \int W_1^g(q) \right) \right).
\end{align}
Now according to Theorem \ref{thm:CZ} of Chen and Zhou, we know that the one-point correlation functions are given by
\begin{equation}
W_1^g (t) = (-1)^{g} \sum_{b \geq 0} \langle \tau_b \Gamma_g(f) \rangle \mathrm{d} \phi_b(t).
\end{equation}
Hence the arbitrary primitives can be taken to be
\begin{equation}
\int W_1^g(t) =  (-1)^{g} \sum_{b \geq 0} \langle \tau_b \Gamma_g(f) \rangle \phi_b(t).
\end{equation}

Using the notation introduced previously, we can also write
\begin{equation}
\log y(t) \frac{\mathrm{d} x(t)}{x(t)} = \left( - \phi_{-1}(t) + \log \left ( \frac{f}{f+1} \right) \right) \frac{\mathrm{d} x(t)}{x(t)}.
\end{equation}
Therefore, noting that the ramification point is at $t = \infty$, what we want to evaluate is the residue
\begin{equation}
F_g =   \frac{1}{2g-2}  \sum_{b \geq 0} \langle \tau_b \Gamma_g(f) \rangle \underset{t=\infty}{\text{Res}} \left( \left( - \phi_{-1}(t) + \log \left ( \frac{f}{f+1} \right) \right) \frac{\mathrm{d} x(t)}{x(t)}  \phi_b(t)\right).
\end{equation}

By definition, for $b \geq 0$ we have that
\begin{equation}
\phi_b(t) = x(t) \frac{\mathrm{d}}{\mathrm{d} x(t)} \phi_{b-1}(t),
\end{equation}
hence
\begin{equation}
F_g =   \frac{1}{2g-2}  \sum_{b \geq 0} \langle \tau_b \Gamma_g(f) \rangle \underset{t=\infty}{\text{Res}} \left( \left( - \phi_{-1}(t) + \log \left ( \frac{f}{f+1} \right) \right) \mathrm{d}\phi_{b-1}(t)\right).
\end{equation}
Since $\phi_{b-1}(t)$ is meromorphic at $t=\infty$ (recall that for $b \geq 1$ it is a polynomial in $t$, while for $b=0$ it vanishes at $t=\infty$), the residue of its differential is necessarily zero. Hence we can forget about the terms involving $ \log \left ( \frac{f}{f+1} \right)$, whose residues all vanish. We get:
\begin{equation}\label{eq:fgtemp}
F_g =    \frac{1}{2-2g}  \sum_{b \geq 0} \langle \tau_b \Gamma_g(f) \rangle \underset{t=\infty}{\text{Res}} \left(  \phi_{-1}(t)  \mathrm{d}\phi_{b-1}(t)\right).
\end{equation}
So we need to evaluate the residue
\begin{equation}
\underset{t=\infty}{\text{Res}} \left(  \phi_{-1}(t)  \mathrm{d}\phi_{b-1}(t)\right)
\end{equation}
for all $b \geq 0$.

For $b=0$, we have
\begin{equation}
\underset{t=\infty}{\text{Res}} \left(  \phi_{-1}(t)  \mathrm{d}\phi_{-1}(t)\right) = \frac{1}{2}\underset{t=\infty}{\text{Res}} \left(\mathrm{d} \left( \phi_{-1}(t)^2\right) \right) = 0,
\end{equation}
Since $\phi_{-1}(t)$ is zero at $t=\infty$. 

For $b=1$, we have:
\begin{align}
\underset{t=\infty}{\text{Res}} \left(  \phi_{-1}(t)  \mathrm{d}\phi_{0}(t)\right) =&-\underset{t=\infty}{\text{Res}} \left(\log \left(1+\frac{1}{ft}\right)\frac{1}{f+1} \right) \mathrm{d} t\\
=& \frac{1}{f(f+1)}.
\end{align}

For $b \geq 2$, we want to evaluate
\begin{align}
R_b := \underset{t=\infty}{\text{Res}} \left(  \phi_{-1}(t)  \mathrm{d}\phi_{b-1}(t)\right) =& - \underset{t=\infty}{\text{Res}} \left(  \mathrm{d }\phi_{-1}(t)  \phi_{b-1}(t)\right)\\
=& - \underset{t=\infty}{\text{Res}} \left(\frac{1}{t(1+f t)}   \phi_{b-1}(t)\right) \mathrm{d} t,
\end{align}
where we used integration by parts. Now we know that
\begin{align}
 \phi_{b-1}(t) =&  \left(\frac{t (t-1) (f t+1)}{f+1} \frac{\mathrm{d}}{\mathrm{d}t}\right)^{b-1} \left(\frac{t-1}{f+1} \right)\\
=&\frac{t (t-1) (f t+1)}{f+1} \frac{\mathrm{d}}{\mathrm{d}t}\left[\left(\frac{t (t-1) (f t+1)}{f+1} \frac{\mathrm{d}}{\mathrm{d}t}\right)^{b-2} \left(\frac{t-1}{f+1} \right)\right],
\end{align}
hence our residue becomes
\begin{align}
R_b =& - \underset{t=\infty}{\text{Res}} \left(\frac{t-1}{f+1} \frac{\mathrm{d}}{\mathrm{d} t} \left[\left(\frac{t (t-1) (f t+1)}{f+1} \frac{\mathrm{d}}{\mathrm{d}t}\right)^{b-2} \left(\frac{t-1}{f+1}  \right)\right] \right)\mathrm{d} t\\
=& - \underset{t=\infty}{\text{Res}} \left( \phi_0(t) \frac{\mathrm{d} \phi_{b-2}(t)}{\mathrm{d}t}  \right) \mathrm{d} t.
\end{align}
But both $\phi_0(t)$ and $\phi_{b-2}(t)$ are polynomial in $t$, and so the one-form $\phi_0(t) \mathrm{d} \phi_{b-2}(t)$  is holomorphic everywhere on $\mathbb{C}_\infty$ except at $t=\infty$, hence by the residue theorem its residue at $t=\infty$ must vanish. Therefore we get that $R_b = 0$ for all $ b \geq 2$.

Putting all this together, we get that the sum in \eqref{eq:fgtemp} collapses onto the $b=1$ term:
\begin{align}
F_g =&    \frac{1}{2-2g}  \langle \tau_1 \Gamma_g(f) \rangle \underset{t=\infty}{\text{Res}} \left(  \phi_{-1}(t)  \mathrm{d}\phi_{0}(t)\right) \\
=& \frac{1}{(2-2g)f(f+1)}  \langle \tau_1 \Gamma_g(f) \rangle,
\end{align}
which is the statement of the lemma.
\end{proof}

To prove Theorem \ref{thm:main}, all that remains is to evaluate the Hodge integral.

\begin{proof}[Proof of Theorem \ref{thm:main}]

By Lemma \ref{lem} we know that
\begin{equation}
F_g =  \frac{1}{(2-2g)f(f+1)}  \langle \tau_1 \Gamma_g(f) \rangle.
\end{equation}
The well known dilaton equation for Hodge integrals \cite{Witten:1991} tells us that
\begin{equation}
\langle \tau_1 \prod_{i=1}^n \tau_{a_i} \rangle_g = (2g-2+n) \langle \prod_{i=1}^n \tau_{a_i} \rangle_g.
\end{equation}
The same result holds when $\lambda$ classes also appear in the Hodge integrals (see for instance \cite{LLZ:2006}). In our case, the dilaton equation implies that
\begin{equation}
F_g = \frac{1}{(2-2g)f(f+1)} (2g-2) \langle \Gamma_g(f) \rangle = -\frac{1}{f(f+1)} \langle \Gamma_g(f) \rangle.
\end{equation}
So we need to evaluate the Hodge integral
\begin{equation}
 \langle \Gamma_g(f) \rangle = \langle \Lambda^\vee_g(1) \Lambda^\vee_g(f) \Lambda^\vee_g(-f-1) \rangle. 
\end{equation}
By definition, we have
\begin{equation}
\Lambda^\vee_g(1) = \sum_{i=0}^g (-1)^i \lambda_i, \qquad  \Lambda_g^\vee(f) = \sum_{i=0}^g (-1)^i \lambda_i f^{g-i}, \qquad \Lambda_g^\vee(-f-1)  =(-1)^g  \sum_{i=0}^g  \lambda_i (f+1)^{g-i}.
\end{equation}
Since we are integrating over the $(3g-3)$-dimensional moduli space $\overline{M}_g$, we only care about the degree $(3g-3)$ class in the product $ \Lambda^\vee_g(1) \Lambda^\vee_g(f) \Lambda^\vee_g(-f-1) $. It is easy to see that this class is given by
\begin{gather}
(-1)^{3g-3} \Big( \lambda_{g-1}^3 \left(-f (f+1) \right) \nonumber \\+ \lambda_g \lambda_{g-1} \lambda_{g-2} \left(f (f+1)^2 - f^2(f+1) + (f+1)^2 - (f+1) +f^2 + f \right)\Big),\label{eq:class}
\end{gather}
where we used the fact that $\lambda_g^2=0$, which follows from Mumford's relation
\begin{equation}\label{Mumford}
\Lambda_g^\vee(t) \Lambda_g^\vee(-t) = (-1)^{g} t^{2g}
\end{equation}
evaluated at $t=0$. Simplifying \eqref{eq:class}, we get that the degree $(3g-3)$ class is given by
\begin{equation}
(-1)^{3g-3}  f(f+1) \left(  - \lambda_{g-1}^3 + 3  \lambda_g \lambda_{g-1} \lambda_{g-2} \right).
\end{equation}
So we have
\begin{equation}
F_g =  (-1)^{g} \langle  - \lambda_{g-1}^3 + 3  \lambda_g \lambda_{g-1} \lambda_{g-2}  \rangle .
\end{equation}
But it also follows from Mumford's relation (from the term in $t^{2}$ in (\ref{Mumford})) that 
\begin{equation}
\langle \lambda_{g-1}^3  \rangle = 2 \langle  \lambda_g \lambda_{g-1} \lambda_{g-2} \rangle,
\end{equation}
hence
\begin{equation}
F_g =  \frac{(-1)^{g}}{2} \langle \lambda_{g-1}^3 \rangle.
\end{equation}
By the result of Faber and Pandharipande \cite{Faber:2000}, we know the value of this Hodge integral:
\begin{equation}
\langle \lambda_{g-1}^3 \rangle = \frac{|B_{2g}| |B_{2g-2}|}{2g (2g-2) (2g-2)!}.
\end{equation}
Therefore, we obtain
\begin{equation}
F_g = (-1)^{g}  \frac{|B_{2g}| |B_{2g-2}|}{2(2g) (2g-2) (2g-2)!},
\end{equation}
as expected from Gromov-Witten theory. 
\end{proof}

\section{Conclusion}

\label{s:conclusion}

In this paper we proved that the free energies computed by the Eynard-Orantin recursion applied to the mirror curve of $\mathbb{C}^3$ reproduce the corresponding Gromov-Witten invariants, as conjectured in \cite{BS:2011}. Our proof relies on previous work of Chen and Zhou \cite{Chen:2009, Zhou:2009}, where the correlation functions are computed in terms of Hodge integrals. 

It would however be nice to obtain a direct proof of the Faber-Pandharipande formula for the $F_g$ without relying on Hodge integrals, using  directly the geometry of the mirror curve. The mirror curve to $\mathbb{C}^3$ is a ``pair of pants'' (genus $0$ with three punctures), hence constitutes the fundamental building block for general mirror curves to toric Calabi-Yau threefolds, as discussed in \cite{BS:2011}. It would be nice to understand how the particular combination of Bernoulli numbers appearing in the $F_g$ naturally comes out of the geometry of a pair of pants. This may be difficult to do for a generic choice of framing, but since we know that the $F_g$ are framing-independent, it may be possible to do explicit calculations for a particularly clever choice of framing ($f=1$ seems to be the natural choice).

\appendix

\section{Connection with the formalism of \cite{Eynard:2011}}  \label{appA}

In a recent paper \cite{Eynard:2011}, Eynard gave a combinatorial interpretation of the coefficients found in the computation of the correlation functions (and of the free energies) for an arbitrary spectral curve  with one ramification point (theorem $3.3$ of \cite{Eynard:2011}). He then applied his general result to the $\mathbb{C}^3$ case, and with the help of several specific identities obtained a formula for the $W_n^g$ which looks very similar to the one proved by Chen and Zhou. In this Appendix we show that the two formulae indeed coincide.

Eynard's formula for the $W_n^g$ is (eq. (7.47), p.$33$):
\begin{equation} W_{\text{Eyn},n}^g(z_1,\dots,z_n)=2^{d_{g,n}}e^{-\tilde{t}_0 \chi_{g,n}} \sum_{d_1,\dots,d_n} \prod_{i} (-1)^{d_i}d\xi_0^{(d_i)}(z_i)\left<\psi_1^{d_1}\dots \psi_n^{d_n} \Gamma_{\text{Eyn}}(f) \right>_g,\end{equation}
with
\begin{equation}
\Gamma_{\text{Eyn}}(f) = \Lambda_{\text{Eyn}}(1)  \Lambda_{\text{Eyn}}(f) \Lambda_{\text{Eyn}}(-f-1).
\end{equation}
In this Appendix for brevity we will only work out the details for the one-point correlation functions, $n=1$, but it is straightforward to generalize the argument. For $n=1$ Eynard's formula reduces to:
\begin{equation}\label{Eyn1}W_{\text{Eyn},1}^g(z)=2^{d_{g,1}}e^{-\tilde{t}_0 \chi_{g,1}} \sum_{b\geq 0}  (-1)^{b}d\xi_0^{(b)}(z)\left<\psi_1^{b}  \Gamma_{\text{Eyn}}(f) \right>_g\end{equation}  
In the above formula we have:
\begin{gather}
d_{g,n}=3g-3+n, \qquad \chi_{g,n}=2-2g-n,\qquad \tilde{t}_0=\ln \sqrt{\frac{f(f+1)}{8}},\\
 \xi_0(z)=\sqrt{\frac{2f}{(1+f)^3}}\frac{1}{z-\frac{f}{f+1}},
\end{gather}
and $\xi_0^{(d)}=\left(\frac{\mathrm{d}}{\mathrm{d} x}\right)^d\xi_0$.
The parametrization of the mirror curve used by Eynard is:
\begin{equation}\label{Para2}
X(z)=e^{-x(z)}=z^f(1-z) \qquad  Y(z)=e^{-y(z)}=z
\end{equation}
Note that the convention used by Eynard for the $\Lambda_{\text{Eyn}}(f)$ are different from ours. More specifically, $\Lambda_{\text{Eyn}}(f)=f^{-g}\Lambda^\vee_g(f)$. Therefore we see that in our notation:
\begin{equation}
\langle\psi_1^{b}\Lambda_\text{Eyn}(1)\Lambda_\text{Eyn}(f)\Lambda_\text{Eyn}(-1-f)\rangle_{g}=(-1)^g\left(f(1+f)\right)^{-g}\langle \tau_b\Gamma_g(f)\rangle
\end{equation} 
The prefactor $e^{-\tilde{t}_0 \chi_{g,1}}$ can also be easily evaluated:
\begin{equation}
e^{-\tilde{t}_0 \chi_{g,1}}=e^{-\frac{1-2g}{2}\ln \frac{f(1+f)}{8}}=\left(\frac{f(1+f)}{8}\right)^{g-\frac{1}{2}}
\end{equation} 
Putting this back into \eqref{Eyn1} leads to:
\begin{equation}\label{Eynn}
W_{\text{Eyn},1}^g(z)=(-1)^g \frac{1}{\sqrt{2 f(1+f)}}\sum_{b\geq 0}(-1)^b \langle \tau_b\Gamma_g(f)\rangle d \xi_0^{(b)}(z)
\end{equation}
The formula obtained by Chen and Zhou is:
\begin{equation}
W_1^g(t) =  (-1)^{g} \sum_{b \geq 0} \langle \tau_b \Gamma_g(f) \rangle \mathrm{d} \phi_b(t).
\end{equation}
So in order to identify Eynard's result with the result of Chen and Zhou, we need to relate the $\xi_0^{(b)}(z)$ to our $\phi_b(t)$. Let us start by relating the two parameterizations of the curve. Eynard's $z$ is related to our $t$ by
\begin{equation}
z = \frac{1}{f+1} \left( f + \frac{1}{t} \right).
\end{equation}
Therefore, we have
\begin{align}
\xi_0^{(0)}(z(t)) =& \sqrt{\frac{2f}{(1+f)^3}}\frac{1}{z(t)-\frac{f}{f+1}} \\
=& \sqrt{\frac{2f}{1+f}} t \\
=&  \sqrt{2 f (f+1)} \phi_0(t) + \sqrt{\frac{2f}{1+f}} .
\end{align}
That is,
\begin{equation}
\mathrm{d} \xi_0^{(0)} = \sqrt{2 f (f+1)} \mathrm{d} \phi_0.
\end{equation}
Now using \eqref{Para2} we have
\begin{align}
\xi_0^{(b)} (z(t))=& \left(\frac{\mathrm{d}}{\mathrm{d} x}\right)^b\xi_0(z(t)) \\
=& \left( \frac{\mathrm{d} z}{\mathrm{d} x} \frac{\mathrm{d}}{\mathrm{d} z} \right)^b \xi_0(z(t))\\
=& \left(\frac{z(1-z)}{(f+1)z - f}  \frac{\mathrm{d}}{\mathrm{d} z} \right)^b \xi_0 (z(t))\\
=& \left(- \frac{t(t f+1)(t-1)}{f+1}    \frac{\mathrm{d}}{\mathrm{d} t}\right)^b \xi_0 (z(t)).
\end{align}
But from \eqref{Phib} we have
\begin{equation}
\phi_b(t) =  \left( \frac{t(t f+1)(t-1)}{f+1}    \frac{\mathrm{d}}{\mathrm{d} t}\right)^b \phi_0(t),
\end{equation}
and since
\begin{equation}
\xi_0(z(t)) = \sqrt{2 f (f+1)} \phi_0(t) + \sqrt{\frac{2f}{1+f}},
\end{equation}
we conclude that (the extra constant term in $\xi_0(z(t))$ does not matter since we are taking derivatives of $\xi_0(z(t))$):
\begin{equation}
\xi_0^{(b)}(z(t)) = (-1)^b \sqrt{2 f (f+1)}  \phi_b(t).
\end{equation}
Putting this back into \eqref{Eynn}, we obtain
\begin{align}
W_{\text{Eyn},1}^g(z(t)) = (-1)^g \sum_{b\geq 0}\langle \tau_b\Gamma_g(f)\rangle d \phi_b(t) = W_1^g(t).
\end{align}
Therefore, Eynard's formula for the correlation functions, in the case of $\mathbb{C}^3$, is precisely equal to the formula proved by Chen and Zhou.

\section{Relation to matrix models} \label{appB}

In this appendix we wish to illustrate a discrepancy between the normalization of matrix models, encoded in the constant terms of the form considered in this paper, and the outcome of the topological recursion. 

The topological recursion formulated in \cite{Eynard:2007} is a solution of the loop equations of matrix models. In general, when applied to the spectral curve of a given matrix model, it reproduces its correlation functions and free energies. However, from the viewpoint of the topological recursion, there is an integration constant ambiguity in the definition of $F_g$, and their particular definition (\ref{eq:fg}) is chosen so that they fulfill certain homogeneity conditions \cite{Eynard:2007}. This does not guarantee that the constant contributions to $F_g$ agree with those of the original matrix model, from which the spectral curve is derived. And in fact, in several cases related to the remodeling conjecture, these contributions differ. 

Such discrepancies for matrix models for the resolved conifold were discussed in \cite{BS:2011}. For example, the normalization factor of the conifold matrix model derived in \cite{OSY} is given by $M(q)$ (with MacMahon function given in (\ref{eq:macmah})); and as its spectral curve coincides with the mirror curve for the conifold, the topological recursion also gives rise to the same normalization, in agreement with Gromov-Witten theory. On the other hand, the normalization of a different conifold matrix model derived in \cite{Eynard:2008ii} does not involve any factor of MacMahon function; however its spectral curve also agrees (up to symplectic transformation) with the mirror curve for the conifold, hence the overall normalization arising from the topological recursion is also given by $M(q)$. So we see that even though, by construction, the two conifold matrix models mentioned above are normalized differently, they give rise to symplectically equivalent spectral curves, and hence the topological recursion produces the same constant contributions in both cases.

It turns out that the above discrepancy can be observed even in a simpler example, related directly to the mirror $\mathbb{C}^3$ geometry which we consider in this paper, and a single MacMahon function. It is known that the MacMahon function is a generating function of plane partitions, and a matrix model encoding such a generating function was constructed in \cite{Eynard:2009nd,OSY}, and its refined version in \cite{Sulkowski:2010ux,Eynard:2011vs}. By construction the partition function of this model is equal to the MacMahon function, and the matrix model takes the form
\begin{equation}  
Z_{\textrm{matrix}} = M(q) = \int \mathcal{D}U \, \det\Big( \prod_{k=1}^{\infty} (1 + q^k e^{-iU}) \Big)  \det\Big( \prod_{k=0}^{\infty} (1 + q^k e^{iU})  \Big)    \label{ZmatrixC3},
\end{equation}
where matrices $U$ of infinite size are integrated over, and $\mathcal{D}U$ is the unitary Vandermonde measure. The spectral curve of the above model was computed explicitly in \cite{Eynard:2011vs} (more generally, in that paper the spectral curve for a refined model with arbitrary $\beta$ was computed; setting $\beta=1$ we get the non-refined spectral curve). This curve reads
\begin{equation}
x^2 = \frac{(1-y)^2}{y}.   \label{C3spectralcurve}
\end{equation}
Taking the square root and comparing with (\ref{eq:mirrorc3}) we see that this is\footnote{Alternatively, by writing the curve (\ref{C3spectralcurve}) as $\big(x-\frac{1-y}{\sqrt{y}}\big)\big(x+\frac{1-y}{\sqrt{y}}\big)=0$, one might interpret it as having two components, each one representing mirror curve for $\mathbb{C}^3$. It would be interesting to study whether one can apply the topological recursion directly to such curves with multiple components. } a mirror curve for $\mathbb{C}^3$ in framing $f=-\frac{1}{2}$. Therefore, by Theorem \ref{thm:main}, the topological recursion applied to this curve computes the free energies as in (\ref{Final}), consistent with Gromov-Witten theory, which arise from the expansion of $M(q)^{1/2}$. Hence we directly see a discrepancy with the original generating function of plane partitions in (\ref{ZmatrixC3}), from which the matrix model was constructed. 

Interestingly, in a sense, we see that the mirror curve for $\mathbb{C}^3$ given in (\ref{eq:mirrorc3}) is naturally associated both to $M(q)$ and $M(q)^{1/2}$, depending on the perspective one is considering. In this context it would be interesting to find a matrix model whose spectral curve would agree with the $\mathbb{C}^3$ mirror curve, and whose partition function by construction would be equal to $M(q)^{1/2}$ and directly agree with Gromov-Witten theory.


\begin{thebibliography}{99}
\bibliographystyle{plain}

\bibitem{Aganagic:2005}
  M.~Aganagic, A.~Klemm, M.~Mari\~no and C.~Vafa,
  ``The Topological Vertex,''
  Commun.\ Math.\ Phys.\  {\bf 254}, 425-478 (2005)
  arXiv:hep-th/0305132.

\bibitem{Bouchard:2009}
  V.~Bouchard, A.~Klemm, M.~Mari\~no and S.~Pasquetti,
  ``Remodeling the B-model,''
  Commun.\ Math.\ Phys.\  {\bf 287}, 117 (2009)
  arXiv:0709.1453 [hep-th].
  %%CITATION = CMPHA,287,117;%%

\bibitem{Bouchard:2010}
  V.~Bouchard, A.~Klemm, M.~Mari\~{n}o and S.~Pasquetti,
  ``Topological open strings on orbifolds,''
  Commun.\ Math.\ Phys.\  {\bf 296}, 589 (2010)
  arXiv:0807.0597 [hep-th].
  %%CITATION = CMPHA,296,589;%%

\bibitem{Bouchard:2009ii}
V.~Bouchard and M.~Mari\~no,
``Hurwitz numbers, matrix models and enumerative geometry,'' 
in \emph{From Hodge Theory to Integrability and tQFT: tt*-geometry,} Proceedings of Symposia in Pure Mathematics, AMS (2008), arXiv:0709.1458v2 [math.AG].

\bibitem{BS:2011}
V.~Bouchard and P.~Su{\l}kowski,
``Topological recursion and mirror curves,''
arXiv:1105.2052v1 [hep-th].

\bibitem{Chen:2009}
  L.~Chen,
  ``Bouchard-Klemm-Mari\~{n}o-Pasquetti Conjecture for C**3,''
  arXiv:0910.3739 [math.AG].
  %%CITATION = ARXIV:0910.3739;%%
	
\bibitem{Eynard:2008ii}                                         
B.~Eynard,
``All orders asymptotic expansion of large partitions,''
J.\ Stat.\ Mech. P07023 (2008).
[arXiv:0804.0381v2 [math-ph]].

\bibitem{Eynard:2009nd}
  B.~Eynard,
  ``A matrix model for plane partitions and (T)ASEP,''
  J.\ Stat.\ Mech.\  {\bf 0910}, P10011 (2009)
  arXiv:0905.0535 [math-ph].

\bibitem{Eynard:2011}
B.~Eynard,
``Intersection numbers of spectral curves,''
	arXiv:1104.0176v2 [math-ph].
  
\bibitem{Eynard:2011vs}
  B.~Eynard and C.~Kozcaz,
  ``Mirror of the refined topological vertex from a matrix model,''
  arXiv:1107.5181v1 [hep-th].

\bibitem{Eynard:2009}
B.~Eynard, M.~Mulase and B.~Safnuk,
``The Laplace transform of the cut-and-join equation and the Bouchard-Mari\~{n}o conjecture on Hurwitz numbers,''
arXiv:0907.5224v3 [math.AG].

\bibitem{Eynard:2007}
B.~Eynard and N.~Orantin,
``Invariants of algebraic curves and topological expansion,''
Comm.\ Numb.\ Theor.\ Phys.\ {\bf 1}, 347-452 (2007)
[arXiv:math-ph/0702045v4].

\bibitem{Eynard:2008}
B.~Eynard and N.~Orantin,
``Algebraic methods in random matrices and enumerative geometry,''
	arXiv:0811.3531v1 [math-ph].

\bibitem{Faber:2000}
C.~Faber and R.~Pandharipande, 
``Hodge integrals and Gromov-Witten theory,''
Invent.\ Math. {\bf 139}, 174-199 (2000)
arXiv:math/9810173v1 [math.AG].

\bibitem{Fay:1970}
J.~Fay, 
``Theta Functions on Riemann Surfaces,'' 
Lecture Notes in Mathematics, {\bf 352}, Springer�-Verlag (1970).

\bibitem{Hori:2000}
  K.~Hori and C.~Vafa,
 ``Mirror symmetry,''
  arXiv:hep-th/0002222.
  %%CITATION = HEP-TH/0002222;%%

\bibitem{Li:2009}
  J.~Li, C.~-C.~M.~Liu, K.~Liu and J.~Zhou,
  ``A Mathematical theory of the topological vertex,''
  Geom.\ Topol.\  {\bf 13}, 527-621 (2009)
  arXiv:math/0408426 [math.AG].

\bibitem{LLZ:2006}
C.~-C.~M.~Liu, K.~Liu and J.~Zhou,
``Mari\~{n}o-Vafa Formula and Hodge Integral Identities,''
J. Algebraic Geom. {\bf 15}, 379-398 (2006)
arXiv:math/0308015v2 [math.AG].

\bibitem{Marino:2008}
  M.~Mari\~no,
  ``Open string amplitudes and large order behavior in topological string theory,''
  JHEP {\bf 0803}, 060 (2008)
  arXiv:hep-th/0612127.
  %%CITATION = JHEPA,0803,060;%%

\bibitem{OSY}
  H.~Ooguri, P.~Su{\l}kowski, M.~Yamazaki,
  ``Wall Crossing As Seen By Matrix Models,''
  Commun. Math. Phys. (2011)
    arXiv:1005.1293 [hep-th].
  
\bibitem{Sulkowski:2010ux}
  P.~Su{\l}kowski,
  ``Refined matrix models from BPS counting,''
  Phys.\ Rev.\  {\bf D83}, 085021 (2011)
  arXiv:1012.3228 [hep-th].
  
\bibitem{Witten:1991}
E.~Witten,
``Two dimensional gravity and intersection theory on moduli space,''
\emph{Surveys in Differential Geometry}, {\bf 1}, 243--310 (1991).

\bibitem{Zhou:2009}
  J.~Zhou,
  ``Local Mirror Symmetry for One-Legged Topological Vertex,''
  arXiv:0910.4320 [math.AG].
  %%CITATION = ARXIV:0910.4320;%%

\bibitem{Zhu:2010}
S.~Zhu, 
``The Laplace transform of the cut-and-join equation of Mari\~no-Vafa formula and its applications,''
arXiv:1001.0618 [math.AG].

\bibitem{Zhu:2011}
S.~Zhu,
``On a proof of the Bouchard-Sulkowski conjecture,''
arXiv:1108.2831v1 [math.AG].




\end{thebibliography}
\end{document}